\documentclass[11pt, reqno]{amsart}
   \usepackage{amsmath}
   \usepackage{amsthm}
   \usepackage{amsfonts}
   \usepackage{amssymb}
   \usepackage[a4paper, margin=0.9in]{geometry}
   \usepackage{enumerate}
     \usepackage{color}
     \usepackage[all]{xy}
     \usepackage{amssymb}
\usepackage{amscd}
\usepackage{graphicx}
\usepackage{braket}
\usepackage{tikz-cd}
\usepackage{hyperref}
\hypersetup{
	colorlinks=true,
	linkcolor=blue,
}
\usepackage{mathtools}

\tikzset{
    slanted/.style={rotate=-90, anchor = south},
    slantedswap/.style={rotate=-90, anchor=north}
}

\usepackage[utf8]{inputenc}

\usepackage[intoc, refpage]{nomencl}
  \newcommand{\twodigitspage}{
   \ifnum \value{page} < 10
   0
   \fi
   \thepage
  }
   \theoremstyle{definition}
   \newtheorem{definition}{Definition}[section]
   \newtheorem{remark}[definition]{Remark}

   \theoremstyle{plain}   
   \newtheorem{proposition}[definition]{Proposition}
   \newtheorem{lemma}[definition]{Lemma}
   \newtheorem{theorem}[definition]{Theorem}
   \newtheorem{corollary}[definition]{Corollary}

   \theoremstyle{definition}

\newcommand{\Gal}{\operatorname{Gal}}

\newcommand{\catname}[1]{\textnormal{{\textsf{#1}}}}

\newcommand{\Tor}{\catname{Tor}}
\newcommand{\T}{\catname{T}}

\newcommand{\fp}{\mathfrak{p}}

 \makeatletter
\newcommand{\colim@}[2]{%
  \vtop{\m@th\ialign{##\cr
    \hfil$#1\operator@font colim$\hfil\cr
    \noalign{\nointerlineskip\kern1.5\ex@}#2\cr
    \noalign{\nointerlineskip\kern-\ex@}\cr}}%
}
\newcommand{\colim}{%
  \mathop{\mathpalette\colim@{\rightarrowfill@\scriptscriptstyle}}\nmlimits@
}
\renewcommand{\varprojlim}{%
  \mathop{\mathpalette\varlim@{\leftarrowfill@\scriptscriptstyle}}\nmlimits@
}
\renewcommand{\varinjlim}{%
  \mathop{\mathpalette\varlim@{\rightarrowfill@\scriptscriptstyle}}\nmlimits@
}
\makeatother

\makeatletter
\newcommand{\Rlim@}[2]{
  \vtop{\m@th\ialign{##\cr
    \hfil$#1\catname{R}\operator@font lim$\hfil\cr
    \noalign{\nointerlineskip\kern1.5\ex@}#2\cr
    \noalign{\nointerlineskip\kern-\ex@}\cr}}%
}
\newcommand{\Rlim}{%
  \mathop{\mathpalette\Rlim@{\leftarrowfill@\scriptscriptstyle}}\nmlimits@
}
\makeatother

  \DeclareMathOperator*{\im}{im}

\newcommand{\QQ}{\mathbb{Q}}
\newcommand{\ZZ}{\mathbb{Z}}

\makeatletter
\def\th@plain{%
  \thm@notefont{}
  \itshape 
}
\def\th@definition{%
  \thm@notefont{}
  \normalfont 
}
\makeatother

\begin{document}

\title[On the coherency of completed group algebras]{On the coherency of\\
completed group algebras}

 \author{David Burns, Yu Kuang, Dingli Liang}

\begin{abstract} We investigate coherency properties of certain completed integral group rings. \end{abstract}

\address{King's College London,
Department of Mathematics,
London WC2R 2LS,
U.K.}
\email{david.burns@kcl.ac.uk}

\address{School of Mathematical Sciences,
Shanghai Jiao Tong University, 
800 Dongchuan Road,
Shanghai 200240,
China}
\email{yu.kuang@sjtu.edu.cn}

\address{King's College London,
Department of Mathematics,
London WC2R 2LS,
U.K.}
\email{dingli.1.liang@kcl.ac.uk}

\vspace*{-1cm}

\maketitle


\section{Introduction}

Following Chase \cite{chase}, and Bourbaki \cite[Chap. 1]{bourbaki}, a ring is said to be (left, respectively right)  `coherent' if every finitely generated (left, respectively right) ideal is finitely presented. The theory of coherent rings is by now well established (for a comprehensive overview see Glaz's book \cite{glaz}) and has important applications, for example in arithmetic geometry.

It is clear that a  Noetherian ring is coherent and also known that any flat direct limit of coherent rings is coherent (cf. [loc. cit., Th. 2.3.3]). However, it can be very difficult to decide whether a given inverse limit of coherent, or even Noetherian, rings is coherent and there still appear to be no general results in this direction. In this note, we consider this problem in the setting of completed group algebras.  

We recall that, for each commutative ring $A$ and profinite group $G$, the completed group algebra is defined  (following \cite{brumer}) to be the inverse limit  
\[ A[[G]] := \varprojlim_U A[ G/U]\]
in which $U$ runs over open normal subgroups of $G$ and the transition map for $U \subseteq U'$ is  the group ring homomorphism $A[G/U] \to A[G/U']$ induced by the projection $G/U \to G/U'$. Such rings arise naturally in several contexts, for example in arithmetic with $\ZZ[[G]]$ acting on the inverse limits of modules (such as class groups, Selmer groups etc.) over a tower of fields within a given Galois extension of number fields of group $G$. 

Our first result, which will be proved in \S\ref{general proof section}, resolves the question of the coherence of $\ZZ[[G]]$ under a mild technical hypothesis on $G$ (and see also Remark \ref{fcr rem}). 

\begin{theorem}\label{general result} If $G$ has a countable basis of neighbourhoods of the identity and a non-torsion Sylow subgroup, then $\ZZ[[G]]$ is neither left nor right coherent.\end{theorem}


The existence of a non-torsion Sylow subgroup is a very mild condition and so the above result applies to most of the groups that arise naturally in arithmetic (cf. Corollary \ref{compact p-adic}). 


With possible arithmetic applications in mind, therefore, it is pertinent to consider the classification of non-coherent rings. Here we focus on the   hierarchy, over natural numbers $n$, of the `$n$-coherence' conditions introduced by 
Costa in \cite{costa}, with `$1$-coherence' being equivalent to the classical notion of coherence. In particular, we recall that $n$-coherent rings, the definition of which is explicitly recalled at the beginning of \S\ref{strong proof section}, have a range of useful properties, including a relatively well-behaved $K$-theory (cf. \cite{ellis}). 

However, despite the weaker nature of these conditions, their verification (for any given $n$) would seem, a priori, to be very difficult (if at all possible) since $\ZZ[[G]]$ is not a compact topological space and so, in any case in which it is not coherent, there are no general techniques for demonstrating finite generation. 

To address these issues, in \S\ref{nakayama section} we introduce a category of `pro-discrete' modules over $\ZZ[[G]]$ and are able to prove a natural analogue of Nakayama's Lemma for this category (see Proposition \ref{nakayama-lemma}). By combining this analogue with well-known results of Brumer \cite{brumer} and Serre \cite{serre}, we shall then, in \S\ref{strong proof section}, deduce the following result.


\begin{theorem}\label{strong coherent thm} If $G$ is a compact $p$-adic analytic group of rank $d$, then $\ZZ[[G]]$ is $(d+3)$-coherent. 
\end{theorem}


Whilst this result is not in all cases best possible (see Remark \ref{last rem}(ii)), establishing any coherency property of this sort for the completed integral group algebras of a general class of profinite groups seems striking (and is, as far as we are aware, without precedent). In addition, such results allow for interesting arithmetic applications. To be more specific, we recall that a stronger version of Theorem \ref{strong coherent thm} was first proved in the special case $G = \ZZ_p$ by Daoud and the first author in \cite{iit}, and we note that some of the techniques used here generalise those of loc. cit. We further recall that the results of \cite{iit} have been used to develop aspects of an arithmetic `integral Iwasawa theory' over $\ZZ[[\ZZ_p]]$, encompassing both concrete new results on the structure of ideal class groups and the formulation (and, in special cases, proof) of a `main conjecture' of integral Iwasawa theory for  $\ZZ_p$-extensions of global fields that strongly refines the classical main conjecture of Iwasawa theory for $\mathbb{G}_m$ (for details see \cite{iit2}). The results of both Theorem \ref{strong coherent thm}, and the more general Propositions \ref{nakayama-lemma} and \ref{upd lemma0}, can similarly contribute towards the development of aspects of an integral Iwasawa theory over more general families of compact $p$-adic analytic extensions of global fields, and we aim to discuss such applications elsewhere.  

\medskip

\noindent{}{\bf Acknowledgements} The first author is very grateful to Alexandre Daoud and Cornelius Greither for insightful discussions concerning this and related projects. 

\section{Coherence results} 

\subsection{The proof of Theorem \ref{general result}}\label{general proof section}

We shall only prove that the stated conditions imply that $\ZZ[[G]]$ is not left coherent (with a completely analogous argument showing that it is not right coherent). 

To do this, we fix a countable basis $\{N_m\}_{m\ge 0}$ of neighbourhoods of the identity of $G$ comprising open normal subgroups $N_m$ with $N_0 = G$ and  $N_{m+1} \subset N_m$ for every $m$. 

We also fix a prime $p$ for which $G$ has a non-torsion Sylow $p$-subgroup $P$ and an element $\pi$ of $P$ of infinite order. We set  
\[ R := \ZZ[[G]]\quad\text{and}\quad \varpi := \pi-1 \in R.\]

For each natural number $m$ we write $\pi_m$ of the image of $\pi$ in the finite group $\Gamma_m := G/N_m$ and $p^{n_m}$ for the order of $\pi_m$ (so that $n_0 = 0$). We assume, as we may (after changing the groups $\{N_m\}_m$ if necessary), that $n_{m+1} > n_m$ for every $m$.  
  We set    
 \[ R_m := \ZZ[\Gamma_m],\,\, T_m := {\sum}_{i=0}^{i=p^{n_m}-1}\pi_m^{i}\in R_m\quad\text{and} \quad\varpi_m := \pi_m-1 \in R_m\]
(so that $R_0 = \ZZ, T_0 = 1$ and $\varpi_0 = 0$). We then define a left $R$-ideal by setting  
\[ I(\varpi) := {\varprojlim}_mR_m\varpi_m \subset {\varprojlim}_mR_m = R,\] 
where the limits are with respect to the natural projection maps $R_m \to R_{m'}$ for $m > m'$. 

Finally, we write $R^p$ and $R_m^p$ for the pro-$p$ completions $\ZZ_p[[G]]$ and $\ZZ_p[\Gamma_m]$ of $R$ and $R_m$ respectively.   

\begin{proposition}\label{ses} The element $\varpi$ is a right non-zero divisor in $R$ and there exists a canonical short exact sequence of (left) $R$-modules 
\[ 0 \to R\varpi \xrightarrow{\subset} I(\varpi) \xrightarrow{\phi_\varpi} R^p/(R^p\varpi + R) \to 0\]
(in which $\phi_\varpi$ is {\em not} induced by the inclusion $I(\varpi) \subset R^p$). 
\end{proposition}

\begin{proof} We write $\Lambda$ for either $R$ or $R^p$, with $\Lambda_m$ denoting the corresponding ring $R_m$ or $R_m^p$. Then, in both cases, there exists an exact commutative diagram 

\begin{equation}\begin{CD}\label{limit diagram}
0 @> >> {\prod}_{m} \Lambda_{m} T_{m} @> \subseteq >> {\prod}_m \Lambda_{m} @> 1\mapsto (\varpi_m)_m >> {\prod}_m \Lambda_{m} \varpi_{m} @> >> 0\\
& & @V(1-\rho_m)_m VV @V(1-\rho_m)_m VV @V(1-\rho_m)_m VV\\
0 @> >> {\prod}_{m} \Lambda_{m} T_{m} @> \subseteq >> {\prod}_m \Lambda_{m} @> 1\mapsto (\varpi_m)_m >> {\prod}_m \Lambda_{m} \varpi_{m} @> >> 0\end{CD}\end{equation}
%
%
\noindent{}in which $\rho_{m}$ denotes the natural projection map $\Lambda_{m} \to \Lambda_{m-1}$ (and its restrictions to both $\Lambda_mT_m$ and $\Lambda_m\varpi_m$). In particular, since $\rho_{m}(T_{m}) = p^{n_{m}-n_{m-1}}\cdot T_{m-1}$ with $n_{m}> n_{m-1}$ and $\rho_{m}(\Lambda_{m}) = \Lambda_{m-1}$, the Snake Lemma applies to this diagram to give an exact sequence 
\begin{equation*} 0 = {\varprojlim}_m \Lambda_mT_m \to \Lambda \xrightarrow{ \lambda\mapsto\lambda\varpi} {\varprojlim}_m \Lambda_{m}\varpi_{m} \to {\varprojlim}_m^1 \Lambda_{m}T_{m} \to {\varprojlim}_m^1 \Lambda_{m} = 0.\end{equation*}
This sequence implies $\varpi$ is a right non-zero divisor in $\Lambda$ and also gives a short exact sequence
\begin{equation}\label{key start} 0 \to \Lambda\varpi \xrightarrow{\subseteq} {\varprojlim}_m \Lambda_{m}\varpi_{m} \to {\varprojlim}_m^1 \Lambda_{m}T_{m} \to 0.\end{equation}

If $\Lambda = R^p$, then the derived limit ${\varprojlim}_m^1 \Lambda_{m}T_{m}$ vanishes since each module $\Lambda_{m}T_{m}$ is finitely generated over $\ZZ_p$ and hence compact. 

To compute ${\varprojlim}_m^1 R_{m}T_{m}$ we write $e_m$ for the idempotent $p^{-n_m}T_m$ of $\QQ[\Gamma_m]$ and $Q_m$ for the quotient of $R_m e_m$ by $R_mT_m$ and use the commutative diagram 

\begin{equation}\begin{CD}\label{limit diagram0}
0 @> >> {\prod}_{m} R_mT_m @> \subseteq >> {\prod}_m R_me_m @> >> {\prod}_m Q_m @> >> 0\\
& & @V(1-\rho_m)_m VV @V(1-\rho_m)_m VV @V(1-\rho'_m)_m VV\\
0 @> >> {\prod}_{m} R_mT_m @> \subseteq >> {\prod}_m R_me_m @> >> {\prod}_m Q_m @> >> 0
\end{CD}\end{equation}
%
%
%
in which each row is the tautological short exact sequence and $\rho_m': Q_m \to Q_{m-1}$ is induced by $\rho_m$. Then, since $\rho_{m+1}(e_{m+1}) = e_m$, by applying the Snake Lemma to this diagram one obtains a short exact sequence 
\begin{equation}\label{key start2} 0 \to {\varprojlim}_m R_{m}e_m \to {\varprojlim}_m Q_{m} \to {\varprojlim}_m^1 R_{m}T_{m} \to 0.\end{equation}
In view of the natural isomorphisms of finite abelian $p$-groups   
\[ Q_m = R_m e_m/(R_mT_m) = R_me_m/(p^{n_m}R_me_m) \cong R^p_me_m/(p^{n_m}R^p_me_m) = R^p_me_m/R^p_mT_m,\]
there are also analogues of the diagrams (\ref{limit diagram0}) in which each term $R_m$ is replaced by $R_m^p$. By passing to the limit over these diagrams and noting ${\varprojlim}_m^1 R^p_{m}T_{m}$ vanishes since each module $R^p_mT_m$ is compact, one obtains an identification 
\begin{equation}\label{ident} {\varprojlim}_m R^p_{m}e_m = {\varprojlim}_m Q_{m}\end{equation}
and hence a short exact sequence 

\begin{equation}\label{key step} 0 \to {\varprojlim}_m R_{m}e_m \xrightarrow{\subset} {\varprojlim}_m R_m^pe_m \to {\varprojlim}_m^1 R_{m}T_{m} \to 0.\end{equation}
In addition, for each $m$, there exists an exact commutative diagram 
\begin{equation*}\begin{CD}
0 @> >> \Lambda_{m+1}\varpi_{m+1} @> \subset >> \Lambda_{m+1} @> 1 \mapsto e_{m+1} >> \Lambda_{m+1}e_{m+1} @> >> 0,\\
& & @V VV @V VV @V VV\\
0 @> >> \Lambda_{m}\varpi_m @> \subset >> \Lambda_{m} @> >> \Lambda_{m}e_{m} @> >> 0\end{CD}\end{equation*}
in which each vertical arrow is induced by $\rho_{m+1}$ and so is surjective. In particular, since $R^p\varpi = {\varprojlim}_m R^p_{m}\varpi_m$ (as a consequence of (\ref{key start}) with $\Lambda = R^p$), by passing to the limit over these diagrams we obtain short exact sequences
\begin{equation}\label{ident2}  0 \to I(\varpi) \to R \to {\varprojlim}_m R_{m}e_{m} \to 0\end{equation}
\begin{equation}\label{ident3} 0 \to R^p\varpi \to R^p \to {\varprojlim}_m R^p_{m}e_{m} \to 0.\end{equation}
These sequences combine with the sequence (\ref{key step}) to induce an identification of the derived limit ${\varprojlim}_m^1 R_{m}T_{m}$ with the quotient $R$-module $R^p/(R^p\varpi + R)$ and then the claimed exact sequence follows directly from (\ref{key start}) with $\Lambda = R$. \end{proof}

In the sequel we fix an element $a\in (\ZZ_p \setminus \QQ) \subset R^p$ 
and write $Q(a)$ for the $R$-submodule of $R^p/(R^p\varpi + R)$ generated by the class of $a$. In the next result we also use the surjective map $\phi_\varpi$ from Proposition \ref{ses}.

\begin{proposition}\label{exp lem} The following claims are valid. 
\begin{itemize}
    \item[(i)] The $R$-module $Q(a)$ is isomorphic to $R/I(\varpi)$.
    \item[(ii)] There exists $x_a\in I(\varpi)$ with $\phi_\varpi(x_a) = a$ and such that the $R$-module $Rx_a$ is free.  
    \end{itemize}
    \end{proposition}

\begin{proof} To prove claim (i) it is enough to show that if $r = (r_m)_m$ is any element of $R = \varprojlim_mR_m$ such that, in $R^p = \varprojlim_mR^p_m$, one has $ra \in R^p\varpi + R$, then for every $m$ one has $r_m \in R_m\varpi_m$. However, if $ra \in R^p\varpi + R$, then for every $m$ there exist elements $b_m$ of $R^p_m$ and $c_m$ of $R_m$ such that $ar_m = (ra)_m = b_m\varpi_m + c_m$ and, upon multiplying this equality on the right by $T_m$ we deduce that 
\[ a r_mT_m = b_m\varpi_m T_m + c_mT_m = c_mT_m.\]
Since $a \notin \QQ$, this equality implies $r_mT_m = 0$ and hence that $r_m \in R_m\varpi_m$, as required.

Next we note that, since $Q(a)$ is non-zero (by claim (i)), any pre-image $x_a$ of the class of $a$ under $\phi_\varpi$ is also non-zero. In particular, if $R$ is a domain (as is the case, by Neumann \cite{neumann}, if $G$ is a torsion-free pro-$p$ group), then the $R$-module $Rx_a$ is automatically free. In the general case, however, the proof of claim (ii) requires more effort. To proceed, for each non-negative integer $i$ we write $a_i$ for the unique integer with $0\le a_i < p^{n_{i+1}-n_i}$ such that  
\[ a = {\sum}_{i \ge 0}a_ip^{n_i} \in \ZZ_p.\]
For integers $j$ with $0 \le j\le m$, we then define elements of $R_m$ by setting 
\[ T_{m,j} := {\sum}_{i=0}^{i=p^{n_j}-1}\pi_m^i \quad \text{and}\quad  y_{a,m} := {\sum}_{j=0}^{j=m-1}a_jT_{m,j}\]
(so $T_{m,0} = 1$ and $T_{m,m} = T_m$). It is then easily checked that the element 

\[ x_a := (\varpi_m y_{a,m})_m \in {\prod}_{m}R_m\]
belongs to $I(\varpi) = \varprojlim_m R_m\varpi_m$ and we aim to verify that this element has the properties stated in claim (ii). 

As a first step, an explicit computation of the connecting homomorphism arising from the diagram (\ref{limit diagram}) shows that the image in $ {\varprojlim}_m^1 \Lambda_{m}T_{m}$ of $x_a$ under the map in (\ref{key start}) is represented by the element 
\begin{equation}\label{step0} (y_{a,m}-\rho_{m+1}(y_{a,m+1}))_m = (-a_mT_m)_m\in {\prod}_m R_mT_m. \end{equation}

In a similar way, an explicit computation of the connecting homomorphism of (\ref{limit diagram0}) shows that this element of ${\varprojlim}_m^1 \Lambda_{m}T_{m}$ is the image under the map in (\ref{key start2}) of the element of $\varprojlim_mQ_m$ that is represented by  
\[ ( \bigl({\sum}_{j=0}^{m-1}a_jp^{n_{j}})e_m\bigr)_m \in {\prod}_m R_me_m.\]
Then, since for each $m$ one has $a \equiv {\sum}_{j=0}^{m-1}a_jp^{n_{j}}$ modulo $p^{n_m}\ZZ_p$, the latter element corresponds under the identification (\ref{ident}) to the element $(ae_m)_m$ of ${\varprojlim}_m R^p_{m}e_m$. Hence, under the isomorphism of ${\varprojlim}_m^1 R_{m}T_{m}$ with 
$R^p/(R^p\varpi + R)$ that is induced by the sequences (\ref{ident}), (\ref{ident2}) and (\ref{ident3}), the element of ${\varprojlim}_m^1 R_{m}T_{m}$ represented by (\ref{step0}) corresponds to the class of $a$. 

This explicit computation has shown that $\phi_\varpi(x_a) = a$ and so, to complete the proof of claim (ii), it is enough for us to prove that the $R$-module $Rx_a$ is free. Hence, since $x_a = \varpi \cdot (y_{a,m})_m$ in ${\prod}_mR_m$ and $\varpi$ is a non-zero divisor of $R$, it is enough for us to show that, for every $m$, the element $y_{a,m}$ is a non-zero divisor of $R_m$. To do this, we fix $m$ and write $\Delta_m$ for the subgroup of $\Gamma_m$ that is generated by $\pi_m$. Then, since $y_{a,m}$ belongs to $\ZZ[\Delta_m]$, it is enough to show that the annihilator in $\QQ^c[\Delta_m]$ of $y_{a,m}$ vanishes. In particular, if for each homomorphism $\chi: \Delta_m \to \QQ^{c\times}$ we write $\chi_\ast$ for the induced ring homomorphism $\QQ^c[\Delta_m] \to \QQ^c$, then it suffices to prove that $\chi_\ast(y_{a,m})\not= 0$ for every $\chi$.

If $\chi$ is trivial, then the sum 
\[ \chi_\ast(y_{a,m}) = {\sum}_{j=0}^{j=m-1}a_j\chi_\ast(T_{m,j}) = {\sum}_{j=0}^{j=m-1}a_jp^{n_j}\]
is non-zero since $0 \le a_j < p^{n_{j+1}-n_j}$ for every $j$. If $\chi$ is non-trivial, and of order $p^d$ say (so $d \le n_m$), then $\chi_\ast(\varpi_m) = \chi(\pi_m)-1$ is non-zero and 
\[ \chi_\ast(\varpi_m)\chi_\ast(y_{a,m}) = \chi_\ast ((\pi_m-1)y_{a,m}) = \chi_\ast\bigl({\sum}_{j=0}^{j=m-1}a_j (\pi_m^{p^{n_j}}-1)\bigr) =
{\sum}_{j\in J_\chi}a_j(\chi(\pi_m)^{p^{n_j}}-1),\]
where $J_\chi$ is the set of integers $j$ with $n_j< d$. It is therefore enough to note that this last sum is non-zero since the elements $\{\chi(\pi_m)^{n_j}-1\}_{j \in J_\chi}$ are linearly independent over $\QQ$ (as $n_j > n_{j'}$ for $j > j'$). 
\end{proof}

To prove Theorem \ref{general result} we now fix an element $x_a$ as in Proposition \ref{exp lem}(ii). Then, since each $R$-module $R\varpi$ and $Rx_a$ is free (the former by the first assertion of Proposition \ref{ses}), the general result of 
%
%
\cite[Cor. 2.1.3]{glaz} implies that the (finitely generated) ideal 
$ R\varpi + Rx_a$ of $R$ is finitely presented if and only if $R\varpi\cap Rx_a$ is finitely generated. It is therefore enough for us to show that $R\varpi\cap Rx_a$ is not finitely generated. To do this, we use the composite isomorphism of $R$-modules   
\[ Rx_a/(R\varpi\cap Rx_a) \cong (R\varpi + Rx_a)/R\varpi \cong Q_a \cong R/I(\varpi)\]
in which the second isomorphism is induced by $\phi_\varpi$ (and the exact sequence in Proposition \ref{ses}) and the third by Proposition \ref{exp lem}(i). In particular, since the $R$-module $Rx_a$ is free of rank one, the displayed isomorphism combines with Schanuel's Lemma \cite[(2.24)]{curtis-reiner} to imply that  $R\varpi\cap Rx_a$ is finitely generated if and only if $I(\varpi)$ is finitely generated.  
 In view of the surjectivity of $\phi_\varpi$, it is therefore enough for us to show that the quotient module $R^p/(R^p\varpi + R)$ is not finitely generated over $R$. 
 
 To establish this, we argue by contradiction and so assume that, for some natural number $t$, the set $\{y_i\}_{1\le i\le t}$ is a set of elements of $R^p$ whose images generate $R^p/(R^p\varpi + R)$ as an $R$-module. Then, writing $\varepsilon: R^p \to \ZZ_p$ for the natural projection map, and noting that $\varepsilon(\varpi) = 0$, it follows that $\{\varepsilon(y_i)\}_{1 \le i\le t}$ is a finite set of generators of the abelian quotient group $\varepsilon(R^p)/\varepsilon(R) = \ZZ_p/\ZZ$ and this is not possible since $\ZZ_p/\ZZ$ is uncountable.  
 
This completes the proof of Theorem \ref{general result}. 


\begin{remark}\label{fcr rem} As a natural weakening of the notion of coherence, a domain is said to be a (left, respectively right) `finite conductor domain' if the intersection of any two of its principal (left, respectively right) ideals is finitely generated (see Glaz \cite{glaz-fcd}, but note that the concept was first considered by Dobbs in  \cite{dobbs}). In particular, by showing that $R\varpi \cap Rx_a$ is not finitely generated over $R$, the above argument implies that, under the conditions of Theorem \ref{general result}, $\ZZ[[G]]$ is not a (left, respectively right) finite conductor domain. \end{remark}

\subsection{Examples} The assumed existence of a non-torsion Sylow subgroup rules out profinite groups such as $(\ZZ/p\ZZ)^\mathbb{N}$ for any prime $p$ and ${\prod}_\ell (\ZZ/\ell\ZZ)$  where $\ell$ runs over any infinite set of primes. However it is satisfied by most of the groups that arise naturally in arithmetic. In particular, Theorem \ref{general result} has concrete consequences such as the following.  

\begin{corollary}\label{compact p-adic} Fix a prime $p$. Then the ring $\ZZ[[G]]$ is neither left nor right coherent in each of the following cases: 
\begin{itemize}
\item[(i)] $G$ is a compact $p$-adic analytic group of positive rank.
\item[(ii)] $G$ is the Galois group of an extension of number fields, or of $p$-adic fields, that contains a $\ZZ_\ell$-subextension for any prime $\ell$.
\item[(iii)] $G$ is a Sylow $p$-subgroup of the absolute Galois group of a number field.
\end{itemize}
\end{corollary}

\begin{proof} To prove claim (i) we recall Lazard \cite{Lazard} has shown that any compact $p$-adic analytic group is isomorphic to a closed subgroup of ${\rm GL}_n(\ZZ_p)$ for some $n$. It is then enough to note that, for any infinite subgroup $G$  of ${\rm GL}_n(\ZZ_p)$ the collection $\{ G\cap ({\rm I}_n + p^m\cdot {\rm M}_n(\ZZ_p))\}_{m \ge 1}$ is a countable basis of neighbourhoods of the identity that comprises open, torsion-free, pro-$p$ subgroups (that are normal in $G$).   

To prove claim (ii) we fix a finite extension $K$ of either $\QQ$ or $\QQ_p$, an algebraic closure $K^c$ of $K$ and a Galois extension $L$ of $K$ in $K^c$, with $G := \Gal(L/K)$, for which there exists an intermediate field $E$ for which $\Gamma := \Gal(E/K)$ is isomorphic to $\ZZ_\ell$. Then, for each natural number $n$, the composite $K(n)$ of all finite extensions $K'$ of $K$ inside $L$ with the property that the absolute value of the discriminant of $K'/\QQ$ is at most $n$ is a finite Galois extension of $K$. In the case of the number fields, respectively $p$-adic fields, this follows directly from the Hermite-Minkowski Theorem (cf. \cite[\S III.2]{neukirch}), respectively \cite[Prop. 14, II, \S5]{lang}. The groups $\{ \Gal(L/K(n))\}_{n \ge 1}$ then give a countable base of neighbourhoods of the identity of $G$. It is then enough to note $G$ is a semi-direct product $\Gal(L/E)\rtimes \Gamma$ that splits (as $\Gamma \cong \ZZ_\ell$), so that any Sylow $\ell$-subgroup of $G$ is not torsion. 

To prove claim (iii) we fix a number field $K$ and a Sylow $p$-subgroup $P$ of $\Gal(K^c/K)$. It is then enough to note that $P$ has a countable base of neighbourhoods of its identity (inherited from the countable base of $\Gal(K^c/K)$ constructed in claim (ii)) and a subgroup that is a free pro-$p$ group on countably many generators (for a proof of the latter fact, see Bary-Soroker et al \cite[\S3]{BSJN}).  
\end{proof}

\section{$n$-coherence results}\label{strong section} 

In this section we continue to use the notation fixed at the beginning of \S\ref{general proof section}, so that $R = \ZZ[[G]]$ and $R_n = \ZZ[\Gamma_n]$ with $\Gamma_n = G/N_n$. 

For a left $R$-module $M$ and homomorphism $\phi$ of such modules, and a non-negative integer $n$, we also set 
\[ M_{(n)} := R_n\otimes_{R}M, \quad \text{and}\quad  \phi_{(n)} := R_n\otimes_{R}\phi,\]
respectively regarded (naturally) as a left $R_n$-module and as a map of left $R_n$-modules. 

\subsection{Nakayama's Lemma}\label{nakayama section} We shall first prove a useful analogue of Nakayama's Lemma for the following category of $R$-modules.  

\begin{definition} An $R$-module $M$ is `pro-discrete' (with respect to the given filtration $\{N_n\}_n$ of $G$) if the natural map $M \to \varprojlim_n M_{(n)}$ is bijective.\end{definition}

\begin{remark} The ring $R = \varprojlim_nR_n$ is itself a pro-discrete $R$-module since $R_{(n)} = R_n$ for every $n$. In general, however, finitely-presented $R$-modules need not be  pro-discrete and the category of pro-discrete $R$-modules need not be abelian (cf. \cite[Rem. 3.12]{iit}). \end{remark} 

 

In the sequel, for any finitely generated left module $N$ over a ring $\Lambda$ we write $\mu_\Lambda(N)$ for the minimal number of generators of $N$.

\begin{proposition}\label{nakayama-lemma}
 If $G$ is a pro-$p$ group, and the $R$-module $M$ is pro-discrete, then the following claims are valid.
\begin{itemize}
\item[(i)] $M$ is finitely generated if and only if it contains a finite subset that, for every $n$, projects to give a set of generators of the $R_n$-module $M_{(n)}$. 
\item[(ii)] If there exists a natural number $d$ such that $\mu_{R_n}(M_{(n)})\le d$ for every $n$, then $M$ is finitely generated and  $\mu_{R}(M) \le \mu_{\ZZ}(M_{(0)}) + d$. 
\end{itemize}
\end{proposition}

\begin{proof} To prove claim (i), it is enough to show that  any finite subset $\{ z_i := (z_{i,n})_n\}_{1\le i\le m}$ of $M \cong \varprojlim_nM_{(n)}$ with the stated property generates $M$ over $R$. To do this, we consider, for each $n$, the exact commutative diagram 
\begin{equation}\label{key diagram}\begin{CD}& & & & \ker(\rho^{\oplus m}_{n+1}) @> \iota_{n+1}' >> \ker(\theta_{n+1}) \\
 @. @. @V VV @V VV\\
          0@> >>  \ker(\iota_{n+1}) @> >>  R_{n+1}^m @> \iota_{n+1} >> M_{(n+1)} @> >> 0 \\
            @. @V \theta'_{n+1}VV @V\rho^{\oplus m}_{n+1} VV @V\theta_{n+1} VV\\
          0@> >>  \ker(\iota_{n}) @> >>  R_{n}^m @> \iota_n >> M_{(n)} @> >> 0.\end{CD} \end{equation}
Here $\iota_n$ is the (assumed to be surjective) map of $R_{n}$-modules that sends the $i$-th element in the standard basis of $R_{n}^m$ to $z_{i,n}$, $\rho_{n+1}$ is the natural map $R_{n+1} \to R_{n}$, $\theta_{n+1}$ is the canonical map and $\theta'_{n+1}$ and $\iota_{n+1}'$ are the respective restrictions of $\rho_{n+1}^{\oplus m}$ and $\iota_{n+1}$. 

Write $J_{n+1}$ for the submodule of $R_{n+1}$ generated by  
$\{h-1:h \in N_{n}/N_{n+1}\}$. Then the map $\rho^{\oplus m}_{n+1}$ is surjective, with kernel equal to the submodule $J_{n+1}^{\oplus m}$ of $R_{n+1}^m$, and it is clear that   
$\ker(\theta_{n+1}) = J_{n+1}\cdot M_{(n+1)}$. It follows that $\iota_{n+1}'$ is surjective and hence, by applying the Snake Lemma to (\ref{key diagram}), that $\theta'_{n+1}$ is  surjective. This last fact implies (via the Mittag-Leffler criterion) that the derived limit ${\varprojlim}^1_n\ker(\iota_n)$ with respect to the maps $\theta'_n$ vanishes. Upon passing to limit over $n$ of the commutative diagrams given by the second and third rows of (\ref{key diagram}), one therefore deduces that the map of $R$-modules 
\[ R^m = \varprojlim_n R_{n}^m \to \varprojlim_n M_{(n)} \cong M\]
that sends the $i$-th element in the standard basis of $R^m$ to $z_i$ is surjective. It follows that $M$ is finitely generated, as required. 
    
To prove claim (ii) we note $R_{0} = \ZZ$ and set $\kappa:= \mu_{\ZZ}(M_{(0)}) \le d$. We show first that, for each $n$, there exists a subset $X_n := \{ x_{i,n}\}_{1\le i\le \kappa}$ of $M_{(n)}$ with the following two properties: 
\begin{itemize}
\item[(P1)]  the $R_n$-submodule $M_n$ of $M_{(n)}$ generated by $X_n$ has finite, prime-to-$p$ index;
\item[(P2)]  for each $n'< n$, the natural map $M_{(n)} \to M_{(n')}$ sends $x_{i,n}$ to $x_{i,n'}$ for every index $i$ and also induces an isomorphism of $R_{n'}$-modules $R_{n'}\otimes_{R_n}M_n \cong M_{n'}$.
\end{itemize}
To establish this we use induction on $n$. For $n=0$ the necessary conditions are satisfied by taking $X_0$ to be any set of generating elements for the (assumed to be finitely generated) abelian group $M_{(0)}$ (so that $M_0 = M_{(0)}$). For the inductive step we fix $n >0$ and assume that suitable sets $X_m$ have been constructed for each $m<n$. For each
index $i$ with $1\le i \le \kappa$ we then fix a pre-image $x_{i,n}$ of $x_{i,n-1}$ under the projection map $\theta_{n}: M_{(n)} \to M_{(n-1)}$, set $X_n := \{ x_{i,n}\}_{1\le i\le \kappa}$ and  write
$M_n$ for the $R_n$-submodule of $M_{(n)}$ generated by $X_n$. It is then clear that 
\[ \ZZ_p\otimes_{\ZZ} \theta_{n}(M_{n}) = \ZZ_p\otimes_{\ZZ} M_{n-1} = \ZZ_p\otimes_{\ZZ} M_{(n-1)} = \ZZ_p\otimes_{\ZZ} \theta_{n}(M_{(n)}),\]
where the second equality is a consequence of (P1) (for $n-1$), and hence that 

\[ \ZZ_p\otimes_{\ZZ} M_{(n)} = \ZZ_p\otimes_{\ZZ} M_{n} + \ZZ_p\otimes_{\ZZ}\ker(\theta_{n}) = \ZZ_p\otimes_{\ZZ} M_{n} +  J_{n}\cdot (\ZZ_p\otimes_{\ZZ} M_{(n)}). \]
Now, since $N_{n-1}/N_{n}$ is a finite $p$-group, the ideal $J_{n}$ belongs to the Jacobson radical of $\ZZ_p\otimes_{\ZZ} R_{n}$ and so the last displayed equality combines with Nakayama's Lemma to imply that $\ZZ_p\otimes_{\ZZ} M_{(n)} = \ZZ_p\otimes_{\ZZ} M_{n}$. It follows that the index of $M_{n}$ in $M_{(n)}$ is finite and prime to $p$, and hence that (P1) is satisfied. The first property in (P2) is also clear for this construction, and the second property is true provided that the natural map 
\[ R_{n'}\otimes_{R_n} M_n \to R_{n'}\otimes_{R_n} M_{(n)} \cong M_{(n')}\]
is injective. However, the kernel of this map is isomorphic to a quotient of the group 
\[ {\Tor}_1^{R_n}(R_{n'}, M_{(n)}/M_n) \cong {\Tor}_1^{\ZZ[N_{n'}/N_{n}]}(\ZZ, M_{(n)}/M_n) \cong H_1(N_{n'}/N_{n},
M_{(n)}/M_n),\]
and the latter group vanishes since $N_{n'}/N_{n}$ is a finite $p$-group whilst the order of $M_{(n)}/M_{n}$ is prime to $p$.  

For each $n >0$ we now consider the exact commutative diagram 
\begin{equation}\begin{CD}\label{key diagram2}
0 @> >> {\prod}_{n} M_n @> (\iota_n)_n >> {\prod}_n M_{(n)} @> >> {\prod}_n Q_n @> >> 0\\
& & @V(1-\theta'_n)_n VV @V(1-\theta_n)_n VV @V(1-\theta''_n)_n VV\\
0 @> >> {\prod}_{n} M_n @> (\iota_n)_n >> {\prod}_n M_{(n)} @> >> {\prod}_n Q_n @> >> 0
\end{CD}\end{equation}
%
%
in which $\iota_n$ denotes the natural inclusion map, $Q_n := {\rm cok}(\iota_n)$, $\theta_n'$ is the restriction of $\theta_n$ and $\theta_n''$ is induced by $\theta_n$. In particular, since the maps $\theta_n'$ are surjective, the Snake Lemma applies to this diagram to give a short exact sequence of $R$-modules 
\begin{equation}\label{reduced ses} 0 \to M' \to M \to Q\to 0,\end{equation}
in which we set $M':= \varprojlim_n M_n$ and $Q:= \varprojlim_n Q_n$ (with the respective limits taken with respect to the maps $\theta_n'$ and $\theta_n''$). In addition, the final assertion of property (P2) implies that $M'$ is a pro-discrete $R$-module and also combines with claim (i) to imply that  $\{ (x_{i,n})_n\}_{1\le i\le \kappa}$ is a generating set for $M'$ and hence that 
\begin{equation}\label{first inequality} \mu_{R}(M') \le \kappa = \mu_{\ZZ}(M_{(0)}).\end{equation} 

To establish $M$ is finitely generated with $\mu_{R}(M) \le \mu_{\ZZ}(M_{(0)}) + d$, we are therefore reduced, via the exact sequence (\ref{reduced ses}) and inequality (\ref{first inequality}), to showing that $\mu_{R}(Q)\le d$. 

In addition, since $M'$ and $M$ are both pro-discrete, the  commutative diagrams (\ref{key diagram2}) can be used to show that, for each $n'<n$, the natural map $R_{n'}\otimes_{R_n}Q_n \to Q_{n'}$ is bijective and hence that $Q$ is a pro-discrete $R$-module. In view of claim (i), we are therefore reduced to constructing a subset $Z$ of $Q \cong \varprojlim_nQ_n$ with the property that, for every $n$, the $R_n$-module $Q_n$ is generated by the projection of $Z$. 

Now, for each $n$, the module $Q_n$ has order prime-to-$p$ and so is naturally a module over $R_n[1/p]= \ZZ[1/p][\Gamma_n]$. In particular, the central idempotent 
\[ e_n := |(N_{n-1}/N_n)|^{-1}{\sum}_{\gamma \in N_{n-1}/N_n}\gamma\]
of $R_n[1/p]$ induces a direct sum decomposition of $R_n$-modules 
\[ Q_n = (1-e_n)Q_n \oplus e_nQ_n \cong (1-e_n)Q_n \oplus (R_{n-1}\otimes_{R_n}Q_n) \cong (1-e_n)Q_n \oplus Q_{n-1}. \]
To inductively construct a suitable subset $Z$ of $Q$, it is therefore enough to show, for every $n$, that $\mu_{R_n}\bigl((1-e_n)Q_n \bigr) \le d$. This in turn follows immediately from the fact that $Q_n$ is quotient of $M_{(n)}$ and, by assumption, one has $\mu_{R_n}(M_{(n)}) \le d$. 

This completes the proof of claim (ii). 
\end{proof}

\subsection{Ddivisibility of Tor-groups}\label{strong proof section0} 

For an abelian group $A$ and natural number $m$ we set $A[m]:= \{a\in A:m\cdot a = 0\}$. We also write $A_{\langle p\rangle}$ for the inverse limit $\varprojlim_{m\in \mathbb{N}} A/p^m$ (with respect to the natural projection maps), and use similar notation for homomorphisms. For a ring $R$, we write ${\rm pd}_R(M)$ for the projective dimension of a left $R$-module $M$. We also recall that, for any natural number $n$, a left $R$-module $M$ is said to be `finitely $n$-presented' if there exists a collection of natural numbers $\{t_i\}_{0\le i\le n}$ and an exact sequence of (left) $R$-modules of the form 
\begin{equation}\label{fixed resolution20}  0 \to \ker(\theta_n) \xrightarrow{\iota} R^{t_{n}} \xrightarrow{\theta_{n}} R^{t_{n-1}} \cdots \xrightarrow{\theta_1} R^{t_0} \xrightarrow{\theta_0}   M \to 0.\end{equation}

The following technical result will be useful for the proof of Theorem \ref{strong coherent thm}. 

\begin{proposition}\label{upd lemma0} Let $G$ be a profinite group and $M$ a finitely generated left $\ZZ[[G]]$-module with the following properties: 
\begin{itemize}
\item[(i)] $M[p] = (0)$. 
\item[(ii)] $M$ is finitely $n$-presented, for some natural number $n$.
\item[(iii)] ${\rm pd}_{\ZZ_p[[G]]}\bigl(M_{\langle p\rangle}\bigr) < n$.
\end{itemize}
Then, for any $\ZZ[[G]]$-module $L$ with $L[p]=(0)$, and any integer $a \ge n$, the higher Tor-group $\Tor_a^{\ZZ[[G]]}(L,M)$ is $p$-divisible. \end{proposition}

\begin{proof} Set $R := \ZZ[[G]]$ and $\Lambda := R_{\langle p\rangle} = \ZZ_p[[G]]$. We first make an easy observation about short exact sequences. For this we note that, if $M_3$ is any $R$-module with $M_3[p]=(0)$, then a short exact sequence of $R$-modules $0\to M_1\to M_2\to M_3\to 0$ gives rise, for each natural number $m$, to an exact commutative diagram
\[ \begin{CD} 
0 @> >> M_1/p^{m+1} @> >> M_2/p^{m+1} @> >> M_3/p^{m+1} @> >> 0\\
& & @V \varrho_{1,m}VV @V \varrho_{2,m}VV @V\varrho_{3,m} VV\\
0 @> >> M_1/p^m @> >> M_2/p^m @> >> M_3/p^m @> >> 0,\end{CD}\]
in which each map $\varrho_{i,m}$ is the natural projection. Then, since $\varrho_{1,m}$ is surjective, the Mittag-Leffler criterion implies that, upon passing to the limit over $m$ of these sequences, one obtains a short exact sequence of $\Lambda$-modules $0\to M_{1,\langle p\rangle}\to M_{2,\langle p\rangle}\to M_{3,\langle p\rangle}\to 0$. 

Turning now to the proof of the stated result, property (ii) allows us to fix an exact sequence of $R$-modules of the form (\ref{fixed resolution20}). Then, under condition (i), this sequence breaks up into a finite collection of short exact sequences in which no occurring term has an element of order $p$. Hence, by applying the above observation to each of these short exact sequences, one deduces firstly that for each  $m$ the induced sequence 
\begin{equation}\label{fixed resolution2}  0 \to \ker(\theta_n)/p^m \xrightarrow{\iota/p^m} (R/p^m)^{t_{n}}  \xrightarrow{\theta_{n}} (R/p^m)^{t_{n-1}} \cdots \xrightarrow{\theta_1} (R/p^m)^{t_0} \to  M/p^m \to 0\end{equation}
is exact and then, upon passing to the limit over $m$, that the induced sequence of $\Lambda$-modules 
\begin{equation}\label{fixed resolution3} 0 \to \ker(\theta_n)_{\langle p\rangle} \xrightarrow{\iota_{\langle p\rangle}} \Lambda^{t_{n}} \xrightarrow{\theta_{n,\langle p\rangle}} \Lambda^{t_{n-1}} \cdots \xrightarrow{\theta_{1,\langle p\rangle}} \Lambda^{t_0} \to M_{\langle p\rangle} \to 0.\end{equation}
is also exact. By using this sequence to compute Tor-groups, one obtains an isomorphism
\begin{equation}\label{first iso} \Tor^{\Lambda}_{n}(L_{\langle p\rangle},M_{\langle p\rangle}) \cong 
\frac{ \ker\bigl( L_{\langle p\rangle}\otimes_{\Lambda}\theta_{n,\langle p\rangle}\bigr)}
{\im\bigl( L_{\langle p\rangle}\otimes_{\Lambda}\iota_{\langle p\rangle}\bigr)}.
\end{equation}
To compute this group, we note that, for each index $i$, the module $L_{\langle p\rangle}\otimes_{\Lambda} \Lambda^{t_i}$ identifies with $(L_{\langle p\rangle})^{t_i} = \varprojlim_m ((L/p^m) \otimes_{R/p^m} (R^{t_i}/p^m))$. In particular, since inverse limits are left exact, this observation (with $i=n$ and $i=n-1$) gives an equality   
\[ \ker\bigl( L_{\langle p\rangle}\otimes_{\Lambda}\theta_{n,\langle p\rangle}\bigr) = 
\varprojlim_m \ker( (L\otimes_R \theta_n)/p^m) = \varprojlim_m \ker((L/p^m)\otimes_{R/p^m}(\theta_n/p^m)),\]
where the limits are taken with respect to the transition maps induced by the projections $(L/p^m)^{t_n} \to (L/p^{m-1})^{t_n}$. 
In a similar way, one finds that there is a corresponding inclusion 
\[\im\bigl( L_{\langle p\rangle}\otimes_{\Lambda}\iota_{\langle p\rangle}\bigr) \subseteq \varprojlim_m \im((L/p^m)\otimes_{R/p^m}\iota_m).\]
 The isomorphism (\ref{first iso}) therefore induces a surjective composite map of $\Lambda$-modules 
\begin{align}\label{key comp} \Tor^{\Lambda}_{n}(L_{\langle p\rangle},M_{\langle p\rangle}) \twoheadrightarrow&\,\,   
\frac{ \varprojlim_m \ker((L/p^m)\otimes_{R/p^m}(\theta_n/p^m))}{\varprojlim_m \im((L/p^m)\otimes_{R/p^m}\iota_m)}\\
\cong&\,\,
\varprojlim_m\frac{\ker((L/p^m)\otimes_{R/p^m}(\theta_n/p^m))}{\im((L/p^m)\otimes_{R/p^m}\iota_m)}\notag\\
\cong&\,\, \varprojlim_m \Tor^{R/p^m}_{n}(L/p^m,M/p^m).\notag\end{align}
Here the first isomorphism follows from the Mittag-Leffler criterion since the projections 
\[ \im((L/p^m)\otimes_{R/p^m}\iota_m)\to \im((L/p^{m-1})\otimes_{R/p^{m-1}}\iota_{m-1})\]
are surjective, and the second is obtained by computing the groups 
$\Tor^{R/p^m}_{n}(L/p^m,M/p^m)$ via the resolutions (\ref{fixed resolution2}).

Next we note that (since $L[p]$ and $M[p]$ vanish) there are short exact sequences 
\[ 0 \to L\xrightarrow{p^m} L \to L/p^m\to 0 \quad\text{and} \quad 0 \to M\xrightarrow{p^m} M \to M/p^m\to 0\]
which combine to give a composite injective homomorphism of abelian groups
\begin{equation}\label{inj map} \Tor_{n}^{R}( L,M)/p^m \hookrightarrow \Tor_{n}^{R}( L/p^m,M/p^m)
\cong \Tor_{n}^{R/p^m}( L/p^m,M/p^m).\end{equation}
Here the isomorphism is induced by the fact that the standard spectral sequence 
\[ \Tor_{b}^{R/p^m}\bigl(L/p^m,\Tor_c^{R}( M, R/p^m )\bigr) \Longrightarrow \Tor_{b+c}^{R/p^m}( L/p^m,M/p^m)\]
collapses on its first page since $\Tor_c^{R}( M, R/p^m )$ vanishes for all $c > 0$ (as  
$\Tor_1^{R}( M, R/p^m )$ is isomorphic to $M[p^m]$). After taking the inverse limit over $m$ of the maps (\ref{inj map}), we deduce from (\ref{key comp}) that $\Tor_{n}^{R}( L,M)_{\langle p\rangle}$ is isomorphic to a subquotient of 
$\Tor^{\Lambda}_{n}(L_{\langle p\rangle},M_{\langle p\rangle})$. 

In particular, since property (iii) implies that $\Tor^{\Lambda}_{n}(L_{\langle p\rangle},M_{\langle p\rangle})$ vanishes, the module $\Tor_{n}^{R}( L,M)_{\langle p\rangle}$ must also vanish and so the group $\Tor_{n}^{R}( L,M)$ is $p$-divisible. 

This proves the stated claim with $a=n$. To prove the same result for all $a>n$, one can then use an induction on $a$. The key point for this is that, if 
$0 \to L'\to F\to L\to 0$ is any short exact sequence of left $R$-modules in which $F$ is free, then one has $L'[p]=(0)$ and also, since $a-1 >n-1\ge 0$, the natural exact sequence 
\[ (0) = \Tor_{a}^{R}( F,M) \to \Tor_{a}^{R}( L,M)\to \Tor_{a-1}^{R}( L',M) \to \Tor_{a-1}^{R}( F,M) =(0)\]
implies $\Tor_{a}^{R}( L,M)$ is isomorphic to $\Tor_{a-1}^{R}( L',M)$.  
\end{proof}

\subsection{The proof of Theorem \ref{strong coherent thm}}\label{strong proof section} We henceforth fix a group $G$ as in Theorem \ref{strong coherent thm}, and continue to set $R := \ZZ[[G]]$. We also now fix natural numbers $n$ and $\{t_i\}_{0\le i\le n}$ and an exact sequence of left $R$-modules of the form (\ref{fixed resolution20}). 
%
%

We recall that Costa \cite{costa} defines $R$ to be `left $n$-coherent' if, for every such sequence, the $R$-module $\ker(\theta_{n})$ is finitely generated. (This property is labelled as `strong left $n$-coherence' by Dobbs et al \cite{dobbsetal}, and more conceptual treatments are given by Zhu \cite{zhu} and  Bravo and  P\'erez \cite{bravo}). We note, in particular, that $R$ is left $1$-coherent if and only if it is left coherent in the classical sense of Chase \cite{chase} and Bourbaki \cite{bourbaki}, and we recall that if $R$ is left $n$-coherent, then it is automatically left $n'$-coherent for every  $n'>n$.

We start by recording a useful technical result.  

\begin{lemma}\label{reduction lemma} If $U$ is an open subgroup of $G$, then $R$ is left $n$-coherent if and only if $\ZZ[[U]]$ is left $n$-coherent. \end{lemma}

\begin{proof} Since the index of $U$ in $G$ is finite, the functor $\ZZ[[G]]\otimes_{\ZZ[[U]]}-$ is flat and a left $\ZZ[[U]]$-module $N$ is finitely generated if and only if the left $\ZZ[[G]]$-module 
$\ZZ[[G]]\otimes_{\ZZ[[U]]}N$ is finitely generated. The stated result is a direct consequence of these facts. 
\end{proof}

Following this result (and the observations made in the proof of Corollary \ref{compact p-adic}(i)), to prove Theorem \ref{strong coherent thm} it is enough for us to show the following: if $G$ is both pro-$p$ and has no element of order $p$, and if $n=d+3$ in  
(\ref{fixed resolution20}), then the $R$-module $K := \ker(\theta_{d+3})$ is finitely generated. Our verification of this fact will depend crucially on the properties of $K$ that are established in claim (ii) of the next result. 

%
%
%
%
 
\begin{proposition}\label{tech lemma} Assume that $G$ is pro-$p$ and has no element of order $p$. 

\begin{itemize}
\item[(i)] The $R$-module $\im(\theta_{d+3})$ is pro-discrete. 
\item[(ii)] The $R$-module $K$ is equal to ${\varprojlim}_m \varrho_m(K)$ and is pro-discrete.
\end{itemize}
\end{proposition}

\begin{proof} For each natural number $a$, there exists a commutative diagram of $R$-modules 
\begin{equation}\label{firstdiagram}
\xymatrix{ 0 \ar[r] & K \ar[dl] \ar[dd]^{\varrho^t_{a+1}} \ar[r]^{\iota} & R^t \ar[dd]^{\varrho^t_{a+1}} \ar[r]^{\theta} & M'\ar[dd] \ar[r] &0\\
 K_{(a+1)}\ar[dr]^{\nu_{a+1}}\\
 0 \ar[r] & \varrho_{a+1}(K) \ar[r]^{\,\,\,\iota_{a+1}}\ar[d]^{\rho_{a+1}^t} & R_{a+1}^t\ar[d]^{\rho_{a+1}^{t}} \ar[r]^{\hskip-0.1truein \theta_{(a+1)}}& M'_{(a+1)}\ar[d] \ar[r] & 0,\\
0 \ar[r] & \varrho_{a}(K) \ar[r]^{\iota_a} & R_{a}^t \ar[r]^{\theta_{(a)}}& M'_{(a)} \ar[r] & 0.
}
\end{equation}
Here we set $\theta = \theta_{d+3}$, $K = \ker(\theta)$, $M' := \im(\theta) = \ker(\theta_{d+2})$ and $t = t_{d+3}$, and write  $\iota$ for the tautological inclusion $K \subseteq R^t$. We also write $\varrho_{a+1}^t$ and $\rho_{a+1}^t$ for the projections $R^t \to R_{a+1}^t$ and $R_{a+1}^t \to R_a^t$ (so that $\varrho_a^t = \rho_{a+1}^t\circ \varrho_{a+1}^t$), $\nu_{a+1}$ for  the canonical map $K_{(a+1)} \to \varrho_{a+1}(K)$ and $\iota_{a+1}$ for the inclusion $\varrho_{a+1}(K) \subseteq R_{a+1}^t$. In addition, all unlabelled arrows in the diagram are the natural projections. Then, as $\nu_{a+1}$ is surjective and $R_{a+1}^t = (R^t)_{(a+1)}$, the commutativity of the diagram implies that the second, and in a similar way third, row is exact. In particular, since 
$\rho_{a+1}^{t}$ is surjective the Mittag-Leffler criterion ensures that, by passing to the inverse limit over $a$ of these diagrams, we obtain an exact commutative diagram 
\begin{equation}\label{firstdiagram1}
\xymatrix{ 0 \ar[r] & K \ar[d]_{\iota'} \ar[r]^{\iota} & R^t \ar[d]^{\cong} \ar[r]^{\theta} & M'\ar[d]^{\mu} \ar[r] &0\\
  0 \ar[r] & {\varprojlim}_a\varrho_a(K) \ar[r]^{\subseteq} & {\varprojlim}_aR_{a}^t \ar[r]^{\hskip-0.1truein(\theta_{(a)})_a}& \varprojlim_a M'_{(a)} \ar[r] & 0,
}
\end{equation}
and hence, by applying the Snake Lemma to this diagram, an exact sequence of $R$-modules 
\[ 0 \to K \xrightarrow{\iota'} {\varprojlim}_a \varrho_a(K) \to M' \xrightarrow{\mu} \varprojlim_a M'_{(a)}\to 0.\]
To simultaneously prove claim (i) and the first assertion of claim (ii), it is thus enough to prove $\mu$ is injective. This is, however, a direct consequence of the commutative diagram 
\[ \begin{CD} M' @> \mu' >> R^{t_{d+1}}\\
 @V\mu VV @\vert \\
\varprojlim_a M'_{(a)} @> ( \mu'_{(a)})_a >> \varprojlim_a R^{t_{d+1}}_{a}\end{CD}\]
in which $\mu'$ denotes the natural inclusion. 

To prove the second assertion of claim (ii) it is then enough to show that the map 
\[ \kappa_{a}: R_{a}\otimes_{R_{a+1}}\varrho_{a+1}(K)\to 
\varrho_{a}(K)\]
that is induced by the surjection $\rho_{a+1}^{t}$ is injective (and hence bijective). For this argument we 
write $\Delta$ for the finite normal subgroup $N_a/N_{a+1}$ of $\Gamma_{a+1}$ and note that the functor 
$R_{a}\otimes_{R_{a+1}}-$ on left $R_{a+1}$-modules identifies with taking $\Delta$-coinvariants. In particular, the second and third rows of the exact diagram (\ref{firstdiagram}) give rise to an exact commutative diagram of $R_a$-modules 
\begin{equation*}\label{first diagram 2}
\minCDarrowwidth 1em\begin{CD}
{\Tor}_1^{R_{a+1}}(R_{a},M'_{(a+1)}) 
@> \tilde\kappa_{a} >> \bigl(\varrho_{a+1}(K)\bigr)_{\Delta} @> (\iota_{a+1})_{\Delta} >> \bigl(R_{a+1}^t\bigr)_{\Delta} @> >> \bigl(M'_{(a+1)}\bigr)_{\Delta} @> >> 0\\
& & @V \kappa_{a} VV @V\cong VV @V VV\\
0 @> >> \varrho_{a}(K) @> \iota_a >> R_{a}^t @> >> M'_{(a)} @> >> 0,\end{CD}
\end{equation*}
which implies $\ker(\kappa_{a})= \ker((\iota_{a+1})_{\Delta}) =  \im(\tilde\kappa_{a})$ is isomorphic to a quotient of the homology group 
 ${\Tor}^{R_{a+1}}_1(R_{a},M'_{(a+1)}) \cong H_1(\Delta,M'_{(a+1)})$. 
In particular, since the exponent of the latter group divides $|\Delta|$ (which is a finite power of $p$), the same is true for the group $\ker(\kappa_{a})$. 

On the other hand, the first row of (\ref{firstdiagram}) induces an isomorphism of $\ker(\nu_{a+1})= \ker(\iota_{(a+1)})$ with ${\Tor}^R_1(R_{a+1},M')$ and hence gives rise to an exact commutative diagram 
\begin{equation*}
\begin{CD}
& & \bigl({\Tor}^R_1(R_{a+1},M')\bigr)_{\Delta} @> >> \bigl(K_{(a+1)}\bigr)_{\Delta} @> (\nu_{a+1})_\Delta >>  \bigl(\varrho_{a+1}(K)\bigr)_{\Delta} @> >> 0\\
& & @V VV @V\cong VV @VV\kappa_{a}V\\
0 @> >> {\Tor}^R_1(R_{a},M') @> >> K_{(a)} @> \nu_a >>  \varrho_{a}(K) @> >> 0.\end{CD}
\end{equation*}
This diagram implies $\ker(\kappa_{a})$ is isomorphic to a quotient of $ {\Tor}^R_1(R_{a},M')$. Hence, since the exponent of 
$\ker(\kappa_a)$ divides $|\Delta|$, to prove $\kappa_a$ is injective it is enough to show $ {\Tor}^R_1(R_{a},M')$, and hence also $\ker(\kappa_a)$, is $p$-divisible. To prove this we first note that the exact sequence (\ref{fixed resolution20}) (with $n=d+3$) induces an isomorphism between ${\Tor}^R_1(R_{a},M') = {\Tor}^R_1(R_{a},\im(\theta_{d+3}))$ and ${\Tor}^R_{d+3}(R_{a},\im(\theta_1))$.

The key point now is that, since $G$ has no element of order $p$, its $p$-cohomological dimension is finite and equal to $d$ (by Serre \cite[Cor. (1)]{serre}). In particular, by applying a result of Brumer \cite[Th. 4.1 with $\Omega = \ZZ_p$]{brumer} in this case, it follows that ${\rm pd}_\Lambda\bigl(\im(\theta_1)_{\langle p\rangle}\bigr)\le d+1$. 

In addition, the sequence (\ref{fixed resolution20}) (with $n=d+3$) implies that the $R$-module $\im(\theta_1)$ is a finitely $(d+2)$-presented. Hence, since neither $\im(\theta_1) \subseteq R^{t_0}$ nor $R_a$ has an element of order 
$p$, we may apply Proposition \ref{upd lemma0} with $M=\im(\theta_1), L = R_a, n = d+2$ and $a= d+3$ in order to deduce that ${\Tor}^R_{d+3}(R_{a},\im(\theta_1))$, and hence also $\ker(\kappa_a)$, is $p$-divisible, as required. 
%
%
%
\end{proof}

In view of Proposition \ref{tech lemma}(ii), we can now apply Proposition \ref{nakayama-lemma}(ii) to the $R$-module $K$ to deduce it is finitely generated  provided that, as $m$ varies, the  quantities $\mu_{R_m}\bigl(\varrho_m(K)\bigr)$ are bounded independently of $m$. By applying the Forster-Swan Theorem (cf. \cite[Th. 41.21]{curtis-reiner}) to each order $R_m$,  the latter condition is then reduced to showing the existence of a natural number $c$ such that, for every $m$ and every prime ideal $\fp$ of $\ZZ$, one has $\mu_{R_{m,\fp}}\bigl( \varrho_m(K)_\fp\bigr) \le c$.

If, firstly,  $\fp \not= p\ZZ$, then $R_{m,\fp} = \ZZ_\fp[\Gamma_m]$ is a maximal, and hence hereditary, $\ZZ_\fp$-order (as $\Gamma_m$ is a finite $p$-group). It follows that  $\varrho_m(K)_\fp$ is a projective $R_{m,\fp}$-submodule of $R_{m,\fp}^{t_{d+3}}$ and hence that $\mu_{R_{m,\fp}}\bigl(\varrho_m(K)_\fp\bigr) \le t_{d+3}$ (see, for example, \cite[Prop. 3.3]{ag}, and note that $R_{m,\fp}$ is a primary $\ZZ_\fp$-algebra, as observed at the beginning of \S3 of loc. cit.). 

If, on the other hand, $\fp = p\ZZ$, then the kernel of the projection $R_{m,\fp} \to R_{0,\fp} = \ZZ_p$ belongs to the Jacobson radical of $R_{m,\fp}$ and so Nakayama's Lemma combines with the isomorphism 
$\,\ZZ_p\otimes_{R_{m},\fp}\varrho_m(K) \cong \varrho_0(K)_{\fp}\,$ induced by the (the argument of) Proposition \ref{tech lemma}(ii) to imply 
that $\mu_{R_{m,\fp}}\bigl( \varrho_m(K)_\fp\bigr) = \mu_{\ZZ_p}(\ZZ_p\otimes_\ZZ\varrho_0(K))$.

Thus, if one takes $c$ to be the maximum of $t_{d+3}$ and $\mu_{\ZZ_p}(\ZZ_p\otimes_\ZZ\varrho_0(K))$, then, for all prime ideals $\fp$ of $\ZZ$, one has $\mu_{R_{m,\fp}}\bigl( \varrho_m(K)_\fp\bigr) \le c$. This completes the proof of Theorem \ref{strong coherent thm}.

\begin{remark}\label{last rem} In this remark, we continue to assume $G$ is a compact $p$-adic analytic group of rank $d$, and consider the possibility of strengthening Theorem \ref{strong coherent thm}.   \

\noindent{}(i) In order to prove, by the same method, $R$ is $(d+2)$-coherent, it would be enough to show, if $n=d+2$ in (\ref{fixed resolution20}), then ${\rm pd}_{\Lambda}\bigl(\im(\theta_1)_{\langle p\rangle}\bigr)< d+1$. This condition is satisfied if ${\bigcup}_{n \in \mathbb{N}}M[p^n]$ has no non-zero $p$-divisible subgroup (as is the case if $M$ is pro-discrete) since then the induced map 
$\im(\theta_1)_{\langle p\rangle} \to R_{\langle p\rangle}^{t_0}$ is injective and so  \cite[Th. 4.1]{brumer} implies 
${\rm pd}_\Lambda\bigl(\im(\theta_1)_{\langle p\rangle}\bigr) \le d$. In general, however, establishing injectivity of all of the possible maps $\im(\theta_1)_{\langle p\rangle} \to R_{\langle p\rangle}^{t_0}$ is, in effect, equivalent to showing $R$ satisfies a variant of the Artin-Rees property relative to the ideal $pR$ and seems difficult. 

\noindent{}(ii) If $d\le 2$, then an alternative approach (that does not rely on \cite{brumer} and \cite{serre}) can be used to improve Theorem \ref{strong coherent thm}. Specifically, if either $d=1$, or $d=2$ and $G$ contains a pro-$p$ meta-procyclic subgroup (in the sense of \cite[Chap. 3, Ex. 10]{ddms}), then a special case of the exact sequence in Proposition \ref{ses} can be used to show directly that, if $n=2d$ in (\ref{fixed resolution20}), then ${\Tor}^R_{2d}(R_{a},\im(\theta_{1}))$ is $p$-divisible, and so (by the above argument) $R$ is $2d$-coherent. This approach underlies the proof of \cite[Th. 1.1]{iit} (for $G = \ZZ_p$), but does not apply in all cases since pro-$p$ compact $p$-adic analytic groups need not have any infinite procyclic normal subgroups (for example, if $m \ge 3$, then results \cite[Th. 1 and Th. 3(ii)]{klingenberg} of Klingenberg imply all infinite normal subgroups of ${\rm SL}_m(\ZZ_p)$ are open). In fact, by closely analysing the finite-presentability of pro-discrete modules, it is also shown in \cite[Th. 1.1 and Prop. 5.2]{iit} that $\ZZ[[\ZZ_p]]$ has weak Krull dimension $2$ (in the sense of Tang \cite{tang}) and, for `most' $p$, is a  $(2,2)$-domain that is neither a $(1,2)$-domain or a $(2,1)$-domain (in the sense of Costa \cite{costa}). However, we do not know the extent, if any, to which such finer structure results generalise. \end{remark}


\begin{thebibliography}{10}

\bibitem{ag} M. Auslander, O. Goldman,
\newblock Maximal orders,
\newblock Trans. Amer. Math. Soc. {\bf 97} (1960) 1-24. 

\bibitem{BSJN} L. Bary-Soroker, M. Jarden, D. Neftin,
\newblock The Sylow subgroups of the absolute Galois group
${\rm Gal}(\QQ)$,
\newblock Adv. Math. {\bf 284} (2015) 186-212.

\bibitem{bourbaki} N. Bourbaki,
\newblock Commutative Algebra,
\newblock Springer, 1989. 

\bibitem{bravo} D. Bravo, M. P\'erez, 
\newblock Finiteness conditions and cotorsion pairs. 
\newblock J. Pure Appl. Algebra {\textbf{221}} (2017) 1249-1267. 


\bibitem{brumer} A. Brumer,
\newblock Pseudocompact Algebras, Profinite Groups and Class Formations,
\newblock J. Algebra {\bf 4} (1966) 442-470.

\bibitem{iit} D. Burns, A. Daoud,
\newblock On Integral Iwasawa Theory, I: preliminaries on $\ZZ[[\ZZ_p]]$,  
\newblock submitted.

\bibitem{iit2} D. Burns, A. Daoud,
\newblock On Integral Iwasawa Theory, II: class groups and main conjectures,
\newblock submitted.

\bibitem{chase} S. U. Chase,
\newblock Direct products of modules,
\newblock Trans. Amer. Math. Soc. {\bf 97} (1960), 457-473.

\bibitem{costa} D. L. Costa,
\newblock Parameterizing families of non-Noetherian rings, 
\newblock Comm. Algebra \textbf{22} (1994)  3997-4011.


\bibitem{curtis-reiner} C. W. Curtis, I. Reiner,
\newblock Methods of representation theory: with applications to finite groups and orders Vol. I and II,
\newblock John Wiley and Sons, 1981.

\bibitem{ddms}  J. D. Dixon, M.P.F. du Sautoy, A. Mann, D. Segal, 
\newblock  Analytic pro-$p$ groups,
\newblock Cambridge Stud. Adv. Math. {\bf 61}, Cambridge University Press, Cambridge, 1999. 

\bibitem{dobbs} D. E. Dobbs, 
\newblock On Going Down for Simple Overrings, 
\newblock Proc. Amer. Math. Soc. {\bf 39}  (1973) 515-519.

\bibitem{dobbsetal} D. E. Dobbs, S-E. Kabbaj, N. Mahdou, 
\newblock $n$-coherent rings and modules,
\newblock In: Commutative ring theory (F\`es, 1995), Lecture Notes in Pure and Appl. Math., {\textbf{185}}, Dekker, New York, 1997, 269–281. 


\bibitem{ellis} E. Ellis, R. Parra,
\newblock $K$-theory of $n$-coherent rings,
\newblock J. Algebra Appl. {\bf 21} (2022) Paper No. 2350007. 


\bibitem{glaz-fcd} S. Glaz,
\newblock Finite conductor rings,
\newblock Proc. Amer. Math. Soc. {\bf 129} (2001) 2833-2843.


\bibitem{glaz} S. Glaz,
\newblock Commutative coherent rings,
\newblock Lect. Notes Math. {\bf 1371},  Springer, Berlin, 2006.

\bibitem{klingenberg} W. Klingenberg, 
\newblock Lineare Gruppen \"uber lokalen Ringen, 
\newblock Amer. J. Math. {\bf 83} (1961) 137-153.


\bibitem{lang} S. Lang,
\newblock Algebraic Number Theory,
\newblock Grad. Text Math. {\bf 110}, Springer, New York, 1986. 

\bibitem{Lazard} M. Lazard,
\newblock Groupes analytiques $p$-adiques,
\newblock Publ. Math. I.H.E.S. {\bf 26} (1965) 5-219.



\bibitem{neukirch} J. Neukirch,
\newblock Algebraic Number Theory,
\newblock Springer Verlag, Berlin, 1999.


\bibitem{neumann} A. Neumann,
\newblock Completed group algebras without zero divisors,
\newblock Arch. Math. {\bf 51} (1988) 496-499.



\bibitem{serre} J-P. Serre, 
\newblock Sur la dimension cohomologique des groupes profinis, 
\newblock Topology {\bf 3} (1965) 413-420. 

\bibitem{tang} G. Tang,
\newblock Weak Krull dimension over commutative rings,
\newblock Advances in Ring Theory, 215–224, World Sci. Publ., Hackensack, NJ, 2005.


\bibitem{zhu} Z. Zhu, 
\newblock On $n$-coherent rings, $n$-hereditary rings and $n$-regular rings,
\newblock Bull. Iranian Math. Soc. {\bf 37} (2011) 251-267. 

\end{thebibliography}
\end{document}